\documentclass[reqno,11pt,a4paper]{amsart}
\usepackage{enumitem}
\usepackage{empheq}
\usepackage{booktabs}

\usepackage{hyperref}
\hypersetup{
  colorlinks,
  linkcolor={blue!50!black},
  citecolor={blue!50!black},
  urlcolor={blue!80!black}
}
\usepackage[colorinlistoftodos, bordercolor=orange, backgroundcolor=orange!20,
linecolor=orange, textsize=scriptsize,textwidth=25mm]{todonotes}

\setlist[enumerate]{labelsep=*, leftmargin=1.5pc}
\setlist[enumerate]{label=\normalfont(\roman*), ref=\roman*}
\usepackage[margin=2.7cm]{geometry}
\newcommand{\Z}{\mathbb{Z}}
\newcommand{\Q}{\mathbb{Q}}
\newcommand{\R}{\mathbb{R}}
\renewcommand{\P}{\mathbb{P}}
\newcommand{\NQ}{N_\Q}
\newcommand{\MQ}{M_\Q}
\newcommand{\conv}{\operatorname{conv}}
\newcommand{\Chi}{\mathcal{X}}
\newcommand{\GL}{\mathrm{GL}}
\newcommand{\IK}{\operatorname{IK}}
\newcommand{\mP}{\operatorname{mult}P}
\newcommand{\mX}{\operatorname{mult}X}
\newcommand{\mult}{\operatorname{mult}}
\newcommand{\Hom}{\operatorname{Hom}}
\newcommand{\Vol}{\operatorname{Vol}}
\newcommand{\relint}{\operatorname{relint}}
\theoremstyle{plain}
\newtheorem{theorem}{Theorem}[section]
\newtheorem{proposition}[theorem]{Proposition}
\newtheorem{lemma}[theorem]{Lemma}
\newtheorem{corollary}[theorem]{Corollary}
\newtheorem{conjecture}{Conjecture}
\theoremstyle{definition}
\newtheorem{definition}[theorem]{Definition}
\newtheorem{example}[theorem]{Example}
\newtheorem*{acknowledgments}{Acknowledgments}
\begin{document}
\author[G.\ Averkov]{Gennadiy Averkov}
\address{BTU Cottbus-Senftenberg\\Platz der Deutschen Einheit 1\\03046 Cottbus\\Germany}
\email{averkov@b-tu.de}
\author[A.\,M.\,Kasprzyk]{Alexander Kasprzyk}
\address{School of Mathematical Sciences\\University of Nottingham\\Nottingham, NG$7$\ $2$RD\\UK}
\email{a.m.kasprzyk@nottingham.ac.uk}
\author[M.\ Lehmann]{Martin Lehmann}
\address{Fakult\"at f\"ur Mathematik\\OVGU Magdeburg\\Universit\"atsplatz 2\\39106 Magdeburg\\Germany}
\author[B.\ Nill]{Benjamin Nill}
\address{Fakult\"at f\"ur Mathematik\\OVGU Magdeburg\\Universit\"atsplatz 2\\39106 Magdeburg\\Germany}
\email{benjamin.nill@ovgu.de}
\title[Bounds on canonical fake weighted projective space]{Sharp bounds on fake weighted projective spaces with canonical singularities}
\maketitle
\begin{abstract}
We give a sharp upper bound on the multiplicity of a fake weighted projective space with at worst canonical singularities. This is equivalent to giving a sharp upper bound on the index of the sublattice generated by the vertices of a lattice simplex containing only the origin as an interior lattice point. We also completely characterise when equality occurs and discuss related questions and conjectures.
\end{abstract}
\section{Introduction}\label{sec:introduction}
A \emph{fake weighted projective space}~$X$ is a~$\Q$-factorial complete normal toric variety of Picard rank one. More concretely,~$X$ is given by the quotient~$Y/G$ of some weighted projective space~$Y=\P(\lambda_0, \ldots, \lambda_d)$ by the action of a finite abelian group~$G$ acting free in codimension one. The order~$|G|$ is called the \emph{multiplicity} of~$X$, denoted~$\mX$, and~$X$ is a weighted projective space if and only if~$\mX=1$ (in which case~$X\cong Y$). If~$X$ has at worst canonical singularities then so does~$Y$. In~\S\ref{sec:proof_of_results} we prove a sharp bound on the multiplicity in this case:

\begin{theorem}\label{main}
Let~$X$ be a~$d$-dimensional fake weighted projective space with at worst canonical singularities.
\begin{enumerate}
\item
If~$d \le 3$ then
\[
\mX \le (d+1)^{d-1}
\]
with equality if and only if~$X\cong\P^d/G$. 
\item
If~$d=4$ then
\[
\mX \le 128
\]
with equality if and only if~$X\cong\P(4,1,1,1,1)/G$.
\item\label{main:dim_ge_5}
If~$d\geq 5$ then
\[
\mX \le 3 (s_{d-1}-1)^2
\]
with equality if and only if~$X\cong\P\left(\frac{3 (s_{d-1}-1)}{s_1},\ldots,\frac{3 (s_{d-1}-1)}{s_{d-2}},1,1,1\right) / G$.
\end{enumerate}
In each case, equality is achieved by a unique $X$.
\end{theorem}
The~$s_i$ in~\eqref{main:dim_ge_5} form the \emph{Sylvester sequence}~\cite[A000058]{OEIS} defined by~$s_1 := 2$ and~$s_i:=s_1 \cdots s_{i-1} + 1$ for~$i \ge 2$. We refer to  Proposition~\ref{main2} and Theorem~\ref{conj} for a precise description of the equality cases in Theorem~\ref{main}.

\section{Background}\label{sec:background}
\subsection{Toric geometry}
We begin by recalling some standard facts from toric geometry. For a concise introduction to toric geometry, see~\cite{Dan78}. Throughout let~$N\cong\Z^d$ denote a lattice of rank~$d$; that is,~$N$ corresponds to the lattice of one-parameter subgroups. A convex lattice polytope~$P\subset N_Q:=N\otimes_\Z\Q$ is said to be \emph{Fano} if:
\begin{enumerate}[label=(\alph*), ref=\alph*]
\item
$P$ is of maximum dimension in the underlying lattice;
\item\label{item:fano_origin}
the origin is contained in the strict interior of~$P$;
\item\label{item:fano_primitive}
the vertices of~$P$ are primitive lattice points.
\end{enumerate}

In addition to being fascinating combinatorial objects in their own right (see~\cite{KN13} for an overview), Fano polytopes are in bijective correspondence with toric Fano varieties. The complete fan in~$N$ generated by the facets of~$P$ -- that is, the \emph{spanning fan} of~$P$ -- corresponds to a projective toric variety~$X_P$ with ample anti-canonical divisor~$-K_X$. Two Fano polytopes~$P$ and~$P'$ correspond to isomorphic varieties~$X_P$ and~$X_{P'}$ if and only if there exists a change of basis of~$N$ sending~$P$ to~$P'$. Hence we regard~$P$ as being defined only up to~$\GL(N)$-equivalence.

Now consider the case when~$P:=\conv\{v_0,\ldots,v_d\}\subset\NQ$ is a Fano simplex. By~\eqref{item:fano_origin} there exists a unique choice of~$d+1$ coprime positive integers~$(\lambda_0,\ldots,\lambda_d)\in\Z_{>0}^{d+1}$ such that~$\lambda_0v_0+\cdots+\lambda_dv_d=0$. We call these integers the \emph{(reduced) weights} of~$P$. Furthermore,~\eqref{item:fano_primitive} implies that any~$d$ of these integers are coprime: the weights are~\emph{well-formed}. See~\cite[\S5]{IF00} for the importance of reduced, well-formed weights in the study of weighted projective space.

\begin{definition}
Let~$P\subset\NQ$ be a~$d$-dimensional Fano simplex with weights~$(\lambda_0,\ldots,\lambda_d)$, and let~$N':=v_0\cdot\Z+\cdots+v_d\cdot\Z$ be the sublattice in~$N$ generated by the vertices of~$P$. The rank-one~$\Q$-factorial toric Fano variety given by the spanning fan of~$P$ is ~$X=\P(\lambda_0,\ldots,\lambda_d)/(N/N')$, where the group~$G:=N/N'$ acts freely in codimension one. We call~$X$ a \emph{fake weighted projective space} of \emph{multiplicity}~$\mX:=|G|$.
\end{definition}

Fake weighted projective spaces have been studied in~\cite{Con02,Buc08,Kas09}, amongst other places. Key to this study is the group~$G$. For example,~\cite{Buc08} showed that~$G$ is equal to the fundamental group~$\pi^1_1(X)$ in codimension one of~$X$. More crudely the multiplicity, given by the index of the sublattice~$N'$ in~$N$, can be used to determine many properties of~$X$. For example, it is an immediate corollary of~\cite[Proposition~2]{BB92} that a fake weighted projective space is a weighted projective space if and only if~$\mX=1$. As a consequence, by restricting the simplex~$P$ to the sublattice~$N'$ one recovers the simplex for~$\P(\lambda_0,\ldots,\lambda_d)$.

\subsection{Terminal and canonical singularities}
Terminal and canonical singularities were introduced by Reid, and play a fundamental role in birational geometry~\cite{Rei80,Rei87}. Terminal singularities form the smallest class of singularities that must be allowed if one wishes to construct minimal models in dimensions three or more. Canonical singularities arise naturally as the singularities occurring on canonical models of varieties of general type, and can be regarded as the limit of terminal singularities. In toric geometry, both terminal and canonical singularities have a particularly elegant combinatorial description.

Throughout let~$M:=\Hom(N,\Z)$ denote the lattice dual to~$N$; that is,~$M$ corresponds to the lattice of characters.  Recall that a toric singularity corresponds to a strictly convex rational polyhedral cone~$\sigma\subset\NQ$~\cite{Dan78}. This cone~$\sigma$ is \emph{terminal} if:
\begin{enumerate}[label=(\alph*), ref=\alph*]
\item\label{item:cone}
the points~$v_1,\ldots,v_k\in N$ corresponding to the primitive generators of the rays of~$\sigma$ are contained in an affine hyperplane~$H_u:=\{v\in\NQ\mid u(v)=1\}$ for some~$u\in\MQ$;
\item
the only points of~$\sigma\cap N$ contained on or under~$H_u$ are the origin and the~$v_i$, i.e.\
\[
\sigma\cap N\cap\{v\in\NQ\mid u(v)\leq 1\}=\{0,v_1,\ldots,v_k\}.
\]
\end{enumerate}
The cone~$\sigma$ is \emph{canonical} if~\eqref{item:cone} holds and:
\begin{enumerate}
\item[(b${}'$)]
the only point of~$\sigma\cap N$ contained strictly under~$H_u$ is the origin, i.e.\
\[
\sigma\cap N\cap\{v\in\NQ\mid u(v)<1\}=\{0\}.
\]
\end{enumerate}

It follows immediately from these descriptions that fake weighted projective spaces with at worst \emph{terminal} singularities correspond to \emph{one-point lattice simplices} (that is,~$P\cap N=\{0,v_0,\ldots,v_d\}$), and that fake weighted projective spaces with at worst \emph{canonical} singularities correspond to lattice simplices containing only the origin as an interior lattice point, which we call \emph{canonical lattice simplices}. By considering the restriction of~$P$ to the sublattice~$N'$ it follows that if~$\P(\lambda_0,\ldots,\lambda_d)/G$ has at worst terminal (respectively, canonical) singularities, then~$\P(\lambda_0,\ldots,\lambda_d)$ has at worst terminal (respectively, canonical) singularities~\cite[Corollary~2.4]{Kas09}. We have the following bound on~$\mX$ in the canonical case:

\begin{corollary}[{\cite[Corollary~2.11]{Kas09}}]\label{cor:kas}
Let~$X$ be a~$d$-dimensional fake weighted projective space with at worst canonical singularities. Let~$(\lambda_0,\ldots,\lambda_d)$ denote the weights of~$X$, ordered such that~$\lambda_d\leq \lambda_i$ for all~$0\leq i<d$. Then
\begin{equation}\label{eq:canonical_bound}
\mX\leq\frac{h^{d-1}}{\lambda_0\cdots\lambda_{d-1}}=\frac{\lambda_d}{h}(-K_Y)^d
\end{equation}
where~$h:=\lambda_0+\cdots+\lambda_d$ and~$Y=\P(\lambda_0,\ldots,\lambda_d)$ is the corresponding weighted projective space.
\end{corollary}

\subsection{Gorenstein singularities}\label{subsec:gorenstein}
When the hyperplane~$H_u$ in~\eqref{item:cone} above corresponds to a lattice point~$u\in M$, the singularity is said to be \emph{Gorenstein}. Gorenstein singularities are automatically at worst canonical. In general a Fano polytope~$P\subset\NQ$ corresponds to a toric Fano variety~$X_P$ with Gorenstein singularities if and only if the \emph{dual} polytope~$P^*:=\{u\in\MQ\mid u(v)\ge -1\text{ for all }v\in P\}$ is also a lattice polytope (here the vertices of~$-P^*$ define the hyperplanes~$H_u$). Such polytopes are called \emph{reflexive}.

Gorenstein toric Fano varieties are of particular importance due to Batyrev's mirror symmetry construction~\cite{Bat94,BB96}. In particular, reflexive simplices have been studied extensively; see, for example,~\cite{Con02,Nil07}. The weights of Gorenstein weighted projective space are well-understood~\cite[Corollary~6B.10]{BR86}: a weighted projective space~$\P(\lambda_0,\ldots,\lambda_d)$ is Gorenstein if and only if~$\lambda_i\mid h$ for all~$0\leq i\leq d$, where~$h:=\lambda_0+\cdots+\lambda_d$. Hence the weights can be expressed in terms of unit partitions and are intimately connected with the Sylvester sequence~\cite{Nil07}. Furthermore, weights are preserved when passing to the dual~\cite[Lemma~5.3]{Con02}: if~$P$ is a reflexive simplex with weights~$(\lambda_0,\ldots,\lambda_d)$ then~$P^*$ is also a reflexive simplex with weights~$(\lambda_0,\ldots,\lambda_d)$.

Let~$X=\P(\lambda_0,\ldots,\lambda_d)/G$ be a Gorenstein fake weighted projective space with corresponding reflexive simplex~$P\subset\NQ$. By considering the dual reflexive simplex~$P^*\subset\MQ$ and noting that restricting to a sublattice in~$N$ corresponds to passing to a coarser lattice in~$M$, we conclude that~$\P(\lambda_0,\ldots,\lambda_d)$ is Gorenstein~\cite[Corollary~2.5]{Kas09}. Furthermore, there exists a sharp bound on the multiplicity~$\mX$ in this case:

\begin{corollary}[{\cite[Proposition~5.5]{Con02}}]\label{cor:conrads}
Let~$X$ be a~$d$-dimensional Gorenstein fake weighted projective space with weights~$(\lambda_0,\ldots,\lambda_d)$. Then
\begin{equation}\label{eq:reflexive_multiplicity}
\mX\,\Big|\,\frac{h^{d-1}}{\lambda_0\cdots\lambda_{d-1}\lambda_d}=\frac{1}{h}(-K_Y)^d
\end{equation}
where~$h:=\lambda_0+\cdots+\lambda_d$ and~$Y=\P(\lambda_0,\ldots,\lambda_d)$ is the corresponding weighted projective space.
\end{corollary}

The right-hand side in~\eqref{eq:reflexive_multiplicity} warrants a little explanation. Let~$S$ be the reflexive simplex corresponding to~$Y$ and define~$X_{S^*}$ to be the Gorenstein fake weighted projective space corresponding to the dual reflexive simplex~$S^*$ (where we exchange the roles of the lattices~$N$ and~$M$). The multiplicity of~$X_{S^*}$ can be easily computed as a ratio of volumes:
\[
\mult{X_{S^*}}=\frac{\Vol{S^*}}{\Vol{S}}=\frac{(-K_Y)^d}{h}.
\]
This is the quantity appearing on the right-hand side of~\eqref{eq:reflexive_multiplicity}. By \cite[Proposition~4.4]{Nil07}, equality in Corollary~\ref{cor:conrads} holds if and only if $X \cong X_{S^*}$.

Notice that the weighted projective spaces occurring in Theorem~\ref{main} are Gorenstein. As will be described in more detail below, equality is attained by the Gorenstein fake weighted projective spaces associated to the dual reflexive simplices $S^*$. In some sense the fact that reflexive simplices maximise the multiplicity should not be too surprising. This fits into the broader picture hinted by the following two theorems where, once again, the occurring weighted projective spaces are Gorenstein.

\begin{theorem}[{\cite[Theorem~3.6]{Kas10}}]\label{thm:vol_dim_3}
Let~$X$ be a~$3$-dimensional toric Fano variety with at worst canonical singularities. Then $(-K_X)^3\leq 72$ with equality if and only if~$X\cong\P(3,1,1,1)$ or $\P(6,4,1,1)$.
\end{theorem}

\begin{theorem}[{\cite[Corollary~1.3]{BKN16}}]\label{reflexive-vol-thm}
Let~$X$ be a~$d$-dimensional toric Fano variety with at worst canonical singularities, where~$d\geq 4$. Then $(-K_X)^d\leq 2(s_d-1)^2$ with equality if and only if~$X\cong\P\left(\frac{2(s_d-1)}{s_1},\ldots,\frac{2(s_d-1)}{s_{d-1}},1,1\right)$.
\end{theorem} 

Motivated by the study of canonical models for Fano threefold with terminal $\Q$-factorial singularities and Picard number one,~\cite{Prok05} proves a result whose statement is very similar to Theorem~\ref{thm:vol_dim_3}: the requirement that $X$ is toric is removed, and replaced with the requirement that $X$ is Gorenstein.

\subsection{Two interesting examples}
One might expect that the denominator in Corollary~\ref{cor:kas} should include~$\lambda_d$, imitating  the Gorenstein case given in Corollary~\ref{cor:conrads}. Surprisingly this does \emph{not} give an upper bound on the multiplicity when one allows non-Gorenstein canonical singularities:

\begin{example}
Let~$X$ be the canonical fake weighted projective space~$\P(15,10,3,2)/(\Z/2\Z)$ of multiplicity two, with corresponding simplex
\[
P=\conv\{0,2e_1,3e_2,10e_3\}-(1,1,1).
\]
This is number~$547\,392$ in~\cite{GRDB}. Note that the bound~\eqref{eq:canonical_bound} is sharp, whilst~$h^2/\lambda_0\cdots\lambda_3=1$. In fact~$\P(15,10,3,2)$ is Gorenstein; the corresponding reflexive simplex~$S$ is self-dual, i.e.~$S^*\cong S$.
\end{example}

It might also be expected that for a reflexive simplex~$S$ containing a one amongst the weights, the dual simplex~$S^*$ uniquely maximises the multiplicity. Again, this is \emph{not} the case:

\begin{example}
Let~$X$ be a canonical fake weighted projective space~$\P(2,2,1,1)/G$ of maximum multiplicity. By Corollary~\ref{cor:kas} we have that~$|G|\leq 9$ and, since~$Y=\P(2,2,1,1)$ is Gorenstein, by Corollary~\ref{cor:conrads} we see that this bound is achieved by~$X=X_{S^*}$, where~$S^*$ is the reflexive simplex dual to the simplex~$S$ associated with~$Y$. That is, by
\[
S^*=\conv\{0,3e_1,3e_2,6e_3\}-(1,1,1).
\]
This is number~$547\,396$ in~\cite{GRDB}. There is a second possibility,~$X=X_P$, given by
\[
P=\conv\{(1,0,0),(0,1,0),(2,5,9),(-4,-7,-9)\}.
\]
This is number~$547\,409$ in~\cite{GRDB}. Since~$P$ is not reflexive, so~$X_P\not\cong X_{S^*}$.
\end{example}

\subsection{The cases of equality}
For~$d \le 4$, the equality cases in Theorem~\ref{main} are unique up to isomorphism. This result is a consequence of the existing classifications in low dimension. More precisely, when~$d=2$ every canonical Fano polytope is reflexive. It is a simple matter to list the possible weights and, by Corollary~\ref{cor:conrads}, compute the maximum multiplicity in each case. The results are summarised in Table~\ref{tab:dim2}. The maximum multiplicity is three, uniquely achieved by the simplex~$S^*$ dual to~$\P^2$.

\begin{table}[tb]
\centering
\begin{tabular}{rccc}
\toprule
Weights&$(1,1,1)$&$(2,1,1)$&$(3,2,1)$\\
$\mult{Y_{S^*}}$&$3$&$2$&$1$\\
\bottomrule
\end{tabular}\vspace{0.5em}
\caption{The possible weights and maximum multiplicities when~$d=2$.}\label{tab:dim2}
\end{table}

When~$d=3$ it is no longer true that every canonical simplex is reflexive. From the classification in~\cite[Table~3]{Kas10} we see that there are~$104$ distinct weights. By Corollary~\ref{cor:kas} we find that a canonical fake weighted projective space~$X$ satisfies~$\mX\leq 16$, and if~$\mX=16$ then~$X$ has weights~$(1,1,1,1)$. Corollary~\ref{cor:conrads} tells us this multiplicity is achieved by the reflexive simplex dual to~$\P^3$, and uniqueness comes from inspecting the classification.

When~$d=4$ the canonical weights are classified in~\cite[Theorem~3.6]{Kas13}. There are~$338\,752$ possibilities. Again, by Corollary~\ref{cor:kas} we find that a canonical fake weighted projective space~$X$ satisfies~$\mX\leq 128$, and if~$\mX=128$ then~$X$ has weights~$(4,1,1,1,1)$. By Corollary~\ref{cor:conrads} we see that this maximum multiplicity is obtained by~$S^*$, where~$S$ is the reflexive simplex associated with~$\P(4,1,1,1,1)$. To prove uniqueness we simply work through the possible Hermite normal forms~$H$ with~$\det{H}=128$, computing~$P=H(S)$ in each case. (In practise one makes use of the action of the automorphism group~$S_4$ on~$S$ to keep the number of Hermite forms that need considering manageable.) We find that~$P$ is a canonical Fano simplex if and only if~$P\cong S^*$.

Combining these results gives:

\begin{proposition}\label{main2}
Let~$P\subset\NQ$ be a~$d$-dimensional canonical lattice simplex. 
\begin{enumerate}
\item If~$d\le 3$ then
\[
\mP \le (d+1)^{d-1}
\]
with equality if and only if~$P\cong\conv\{0,(d+1) e_1, \ldots, (d+1)e_d\}-(1,\ldots, 1)$.
\item If~$d=4$ then
\[
\mP \le 128
\]
with equality if and only if~$P \cong\conv\{0,2e_1,8e_2,8e_3,8e_4\}-(1,1,1,1)$. 
\end{enumerate}
\end{proposition}

We prove uniqueness when~$d\ge5$ in~\S\ref{sec:proof_of_results}.

\begin{theorem}\label{conj}
Let~$P\subset\NQ$ be a~$d$-dimensional canonical lattice simplex, where~$d\ge 5$. Then
\[
\mP \le 3 (s_{d-1}-1)^2
\]
with equality if and only if 
\[
P\cong\conv\{0,s_1 e_1,\ldots,s_{d-2} e_{d-2},3 (s_{d-1}-1) e_{d-1},3 (s_{d-1}-1) e_d\}-(1,\ldots,1).
\]
\end{theorem}

Let us notice again that the reflexive simplices described in Proposition~\ref{main2} and Theorem~\ref{conj} are precisely the dual reflexive simplices~$S^*$ described in~\S\ref{subsec:gorenstein} and~\cite[Proposition~4.4]{Nil07}.

\subsection{A remark on volume maximisers}

Let~$Y=\P(\lambda_0,\ldots,\lambda_d)$ have at worst canonical singularities. Is there a sharp upper bound for~$h:=\lambda_0+\cdots+\lambda_d$ (or, equivalently, $\Vol{P}$, where $P\subset\NQ$ is the associated canonical lattice simplex)? This is an open question; the best we can offer is a consequence of~\cite[Table~3]{Kas10} and~\cite[Theorem~3.6]{Kas13}:

\begin{enumerate}
\item\label{item:canonical_h_3}
if $d=3$ then $h\leq 66$ with equality if and only if $Y\cong\P(33,22,6,5)$;
\item\label{item:canonical_h_4}
if $d=4$ then $h\leq 3486$ with equality if and only if $Y\cong\P(1743,1162,498,42,41)$.
\end{enumerate}

It is also natural to ask which polytopes maximise the volume amongst all $d$-dimensional canonical lattice polytopes $P\subset\NQ$. The existing classifications in dimensions two and three~\cite{Rab89,Kas10} give:

\begin{enumerate}
\item
if $d=2$ then $\Vol{P}\leq 9$ with equality if and only if $P\cong\conv\{0,3e_1,3e_2\}-(1,1)$;
\item
if $d=3$ then $\Vol{P}\leq 72$ with equality if and only if
\[
P\cong\conv\{0,2e_1,6e_2,6e_3\}-(1,1,1)\quad\text{ or }\quad\conv\{0,2e_1,3e_2,12e_3\}-(1,1,1).
\]
\end{enumerate}

Here, as throughout, the normalization of the volume is such that the convex hull of an affine lattice basis has volume~$1$. In the case of canonical lattice simplices (that is, allowing \emph{fake} weighted projective spaces), we have the following sharp bound on the volume:

\begin{theorem}[\cite{AKN15}]\label{thm:simplex_vol}
Let $P\subset\NQ$ be a $d$-dimensional canonical lattice simplex, where $d\geq 4$. Then
\[
\Vol{P}\leq 2(s_d-1)^2
\]
with equality if and only if $P\cong\conv\{0,s_1e_1,\ldots,s_{d-1}e_{d-1},2(s_d-1)e_d\}-(1,\ldots,1)$.
\end{theorem}

A consequence of Theorem~\ref{reflexive-vol-thm} is an identical (sharp) bound for reflexive polytopes: simply replace ``canonical lattice simplex'' in the statement of Theorem~\ref{thm:simplex_vol} with ``reflexive polytope'' (see~\cite[Corollary~1.2]{BKN16}). It seems plausible that this is also the volume maximiser amongst all $d$-dimensional canonical lattice polytopes, where $d\geq 4$, although currently there is no direct evidence to support this.

\subsection{Analogous results for terminal singularities}
So far the focus has been on \emph{canonical} singularities. These questions have analogues when we restrict to \emph{terminal} singularities (and the corresponding one-point lattice simplices), however in this case far less is known. Indeed, it seems unlikely that  the methods used to prove Theorems~\ref{main} and~\ref{conj} will apply.

First, is there a sharp upper bound for~$h$? As with the canonical case, this is still wide-open. What we do know -- from the classifications~\cite[Table~4]{Kas06} and~\cite[Theorem~3.5]{Kas13} -- is that the (sharp) upper bounds in low dimensions grow significantly more slowly than in the canonical case:

\begin{enumerate}
\item\label{item:terminal_h_3}
if $d=3$ then $h\leq 19$ with equality if and only if $Y\cong\P(7,5,4,3)$;
\item\label{item:terminal_h_4}
if $d=4$ then $h\leq 881$ with equality if and only if $Y\cong\P(430,287,123,21,20)$.
\end{enumerate}
The corresponding one-point lattice simplex in~\eqref{item:terminal_h_4} is also a volume maximiser amongst all four-dimensional one-point lattice simplices~\cite[Theorem~3.12]{Kas13}. This is not the case in dimension three:

\begin{example}
Let $P\subset\NQ$ be a three-dimensional one-point lattice simplex. Then $\Vol{P}\leq 20$ with equality if and only if $P=\conv\{((1,0,0),(0,1,0),(1,2,5),(-2,-3,-5)\}$. This is number $547\,383$ in~\cite{GRDB}, and corresponds to the fake weighted projective space $X=\P^3/G$ with $\mX=5$ (see~\cite[Example~1.4]{Kas09}).
\end{example}

Next, can we give a sharp bound for the anti-canonical degree? Here we have a likely candidate:
\begin{conjecture}[{cf.~\cite[Lemma~3.7]{Kas13}}]
Let~$X$ be a~$d$-dimensional fake weighted projective space with at worst terminal singularities, where~$d\geq 2$. Then
\[
(-K_X)^d\leq\frac{s_d^d}{(s_d-1)^{d-2}}
\]
with equality if and only if~$X\cong\P\left(\frac{s_d-1}{s_1},\ldots,\frac{s_d-1}{s_{d-1}},1,1\right)$.
\end{conjecture}

Finally, what is the analogous statement to Theorem~\ref{main} in the terminal case? That is, is there a sharp upper bound on the multiplicity of a~$d$-dimensional fake weighted projective space~$X$ with at worst terminal singularities? We can only offer the results of~\cite{Kas06} and~\cite[Theorem~3.12]{Kas13}:
\begin{enumerate}
\item
if $d=3$ then $\mX\leq 5$ with equality if and only if $X_P\cong\P^3/G$ corresponding to the one-point lattice simplex
\[
P=\conv\{(1,0,0),(0,1,0),(1,2,5),(-2,-3,-5)\};
\]
\item
if $d=4$ then $\mX\leq 41$ with equality if and only if $X_P\cong\P^4/G$ corresponding to the one-point lattice simplex
\[
P=\conv\{(1,0,0,0),(0,1,0,0),(0,0,1,0),(3,10,28,41),(-4,-11,-29,-41)\}.
\]
\end{enumerate}
Notice that these bounds satisfy $\mX\leq s_d-2$.

\section{Proof of Theorem~\ref{main} and Theorem~\ref{conj}}\label{sec:proof_of_results}
As mentioned above, for~$d\le 4$ the result can be computationally verified. Hence we may assume~$d \ge 5$. We follow an ansatz that has already been successfully used in~\cite{AKN15}. Therefore, from now on we will use the notation used therein. Let~$P$ be a~$d$-dimensional canonical lattice simplex. We set~$n := d+1$, so~$n \ge 6$. Up to reordering the vertices of the of~$P$, we may assume that~$\beta_1 \ge \ldots \ge \beta_{n} >0$ are the barycentric coordinates of the origin with respect to the vertices $v_1, \ldots, v_{n}$ of~$P$, i.e., 
\[\sum_{i=1}^{n} \beta_i v_i = 0, \quad \sum_{i=1}^{n} \beta_i = 1.\]

The following lemma is a consequence of Corollary~2.11 in~\cite{Kas09}.

\begin{lemma}In this setting, 
\[\mP \le \frac{\beta_n}{\beta_1 \cdots \beta_{n-1}}.\]
\label{lemma-bound}
\end{lemma}

\begin{proof}
Let~$X$ be the fake weighted projective space corresponding to~$P$, and~$\P(\lambda_0, \ldots, \lambda_d)$ its associated weighted projective space. As~$\lambda_d \ge 1$, we get from Corollary~\ref{cor:kas} 
\[\mP \le \frac{h^{d-1}}{\lambda_0 \cdots \lambda_{d-1}} \le \frac{\lambda_d h^{d-1}}{\lambda_0 \cdots \lambda_{d-1}}.\]
As is well-known (see~\cite{Nil07,Kas09}), we have~$(\beta_1, \ldots, \beta_n) = (\lambda_0/h, \ldots, \lambda_d/h)$, so the previous equation gives the desired result.
\end{proof}

As in~\cite{AKN15}, let us turn the originally geometric question into an optimization problem. We recall the following notation from~\cite{AKN15}. Let~$\Chi^n$ denote the set of~$n$-tuples $(x_1, \ldots, x_n) \in \R^n$ which fulfil the conditions 
\begin{empheq}{align}
& x_1 + \ldots + x_n = 1, &\; \label{cond_sum_one}\\
& 1 \ge x_1 \ge \ldots \ge x_n \ge 0, &\;\label{cond_ordered}\\
& x_1 \cdots x_j \le x_{j+1} + \ldots + x_n \quad \forall \; j \in \{1, \ldots, n-1\}. &\; \label{cond_ps_ineq}
\end{empheq}

It is the key result, Theorem 1.1 in~\cite{Ave12}, that~$(\beta_1, \ldots, \beta_{n})$ belongs to the set~$\Chi^n$. 
Given a continuous function~$f: \Chi^n \rightarrow \R$, we denote by~$\IK^n(f)$ the optimization problem of minimizing~$f$ over the set~$\Chi^n$. 
Lemma 4.2(b) in~\cite{AKN15} shows that all optimal solutions of~$\IK^n(x_1 \cdots x_{n-1}  x_n^{-1})$ are of the form
\[\left(\frac{1}{s_1}, \ldots, \frac{1}{s_{l-1}}, \frac{1}{(n-l+1)(s_l - 1)}, \ldots, \frac{1}{(n-l+1)(s_l - 1)} \right)\]
for some~$l \in \{1, \ldots, n\}$. Hence, by considering again the reciprocal we see that~$\mP$ is bounded from above by the maximum of  
\[f_n (l) := \frac{s_1 \cdots s_{l-1} (n-l+1)^{n-l} (s_l-1)^{n-l}}{(n-l+1)(s_l-1)}= (n-l+1)^{n-l-1} (s_l-1)^{n-l}\]
for~$l \in \{1, \ldots, n\}$. 

For~$n=6$ we can directly compute the maximum and see that it is attained uniquely at~$(n,l)=(6,4)$. 
For~$n \ge 7$, it is convenient to let~$r := n-l$, so,~$0 \le r \le n-1$. As we have~$f_n(l) = (r+1)^{r-1} (s_{n-r}-1)^{r}$ and~$f_n(n-2) = 3 (s_{n-2}-1)^{2}$, the upper bound in Theorem~\ref{main} results from the following result.

\begin{lemma}Let~$n \ge 6$, and~$r \in \{0, \ldots, n-1\}$. Then 
\[(r+1)^{r-1} (s_{n-r}-1)^r \le 3 (s_{n-2}-1)^2,\]
with equality if and only if~$r=2$.
\label{lemma1}
\end{lemma}

\begin{proof}
We will proceed by induction on~$n$. As mentioned above it holds for~$n=6$, so let~$n \ge 7$. 
For~$r=0$ we get a strict inequality. For~$r=1$, the left side gives~$s_{n-1}-1 = s_{n-2} (s_{n-2} - 1)$ which is also strictly smaller than~$3 (s_{n-2}-1) (s_{n-2}-1)$. 
For~$r=2$ one has equality. Let~$r \ge 3$. By induction the statement holds for~$(n-1,r-1)$, i.e., 
\[r^{r-2}(s_{n-r}-1)^{r-1}\le 3 (s_{n-3}-1)^2,\]
thus
\[(r+1)^{r-1}(s_{n-r}-1)^r \le 3 (s_{n-3}-1)^2 (s_{n-r}-1)\left(\frac{r+1}{r}\right)^{r-1} r.\]
Becaue of~$\left(\frac{r+1}{r}\right)^{r-1} < \left(\frac{r+1}{r}\right)^r < e$, it suffices to show
\[(s_{n-3}-1)^2 (s_{n-r}-1) e r \le (s_{n-2}-1)^2 = s_{n-3}^2 (s_{n-3}-1)^2,\]
equivalently,
\[(s_{n-r}-1)e r \le s_{n-3}^2.\]
This holds because of the following inequalities 
\[(s_{n-r}-1)e r \stackrel{r \ge 3}{\le} (s_{n-3}-1)e(n-1) \stackrel{n \ge 7}{\le} (s_{n-3}-1)s_{n-3}<s_{n-3}^2.\]
\end{proof}

Let us now consider the equality case, i.e., let $P$ be a~$d$-dimensional canonical lattice simplex with $\mP=3 (s_{d-1}-1)^2$ and $d \ge 5$. It follows from the considerations above and Lemma~\ref{lemma1} that (as $r=2$, so $l=n-2=d-1$) the barycentric coordinates of the origin are completely determined:
\begin{equation}
(\beta_1, \ldots, \beta_{d+1}) = \left(\frac{1}{s_1}, \ldots, \frac{1}{s_{d-2}}, \frac{1}{3(s_{d-1} - 1)}, \frac{1}{3(s_{d-1} - 1)}, \frac{1}{3(s_{d-1} - 1)} \right).
\label{eq:bary}
\end{equation}

Let us recall that two lattice polytopes $Q,Q'$ in $N_\Q$ are considered {\em isomorphic} (or {\em unimodularly equivalent}) if $Q' = A Q + w$ for some $A \in \GL(N)$ and $w \in N$. The next result is a modification of \cite[Lemma~5.2]{AKN15}.

\begin{lemma}
Let $d \ge 5$, and $S$ a~$d$-dimensional canonical lattice simplex where the barycentric coordinates of the origin with respect to the vertices of $S$ are given by \eqref{eq:bary}. Then $S$ is isomorphic to a lattice simplex $\conv\{s_1 e_1, \ldots, s_{d-2} e_{d-2},a,b,c\}$ for $a,b,c \in N$ with $a+b+c=0$ that contains $e_1 + \cdots + e_{d-2}$ as its unique interior lattice point.
\label{lemma2}
\end{lemma}

\begin{proof}
Let the vertices of $S$ be given as $p_1, \ldots, p_{d-2}, u,v,w$, thus,
\begin{equation}
0 = \left(\sum_{i=1}^{d-2} \frac{p_i}{s_i}\right) +\frac{1}{3(s_{d-1}-1)} (u+v+w).
\label{eq:verts}
\end{equation}
Multiplying both sides by $3$ and plugging in $\frac{1}{s_{d-1}-1} = 1 - \sum_{i=1}^{d-2}\frac{1}{s_i}$ implies that 
\[\sum_{i=1}^{d-2}\frac{3p_i-(u+v+w)}{s_i} = -(u+v+w) \in N.\]
As $s_1, \ldots, s_{d-2}$ are relatively prime, applying \cite[Proposition~5.1]{AKN15} coordinatewise yields that for all $i = 1, \ldots, d-2$
\[\frac{3p_i-(u+v+w)}{s_i} \in N.\]
In particularly, for $i=2$ we deduce, since $s_2=3$, that $p_2-\frac{u+v+w}{3} \in N$, so $x := \frac{u+v+w}{3} \in N$. Hence, \eqref{eq:verts} implies that $R := \conv\{p_1, \ldots, p_{d-2},x\}$ is a $(d-2)$-dimensional canonical lattice simplex whose origin has the the barycentric coordinates $\left(\frac{1}{s_1}, \ldots, \frac{1}{s_{d-2}}, \frac{1}{s_{d-1} - 1}\right)$ with respect to the vertices of $R$. Now, \cite[Theorem~2.1(c)]{AKN15} implies that $R$ is isomorphic to $\conv\{s_1 e_1, \ldots, s_{d-2} e_{d-2},0\}$ with unique interior lattice point $e_1 + \cdots + e_{d-2}$. The statement of the lemma follows now from extending this unimodular equivalence to $S$.
\end{proof}

Applying this lemma, we may assume $P = \conv\{s_1 e_1, \ldots, s_{d-2} e_{d-2},a,b,c\}$ for $a,b,c \in N$ with $a+b+c=0$ and that $p := e_1 + \cdots + e_{d-2}$ is the unique interior lattice point of $P$. Again, we consider the $(d-2)$-dimensional simplex $R := \conv\{0,s_1 e_1, \ldots, s_{d-2} e_{d-2}\}$ with $\relint(R) \subset \relint(P)$. As $R-p$ is reflexive of multiplicity $1$ and minimal weight $1$ (cf. \cite[Proposition~4.4]{Nil07}) we have that $p,s_1 e_1, \ldots, s_{d-2} e_{d-2}$ forms an affine lattice basis of $\Z e_1 + \dots + \Z e_{d-2}$. Let us denote the projection of $\R^d$ to the last two coordinates by $\pi$, and define $a' := \pi(a), b':=\pi(b), c':=\pi(c)$. The image $P' := \pi(P)$ is a triangle with vertices $a',b',c'$. By our assumption on the multiplicity of $P$, we see that the index of the sublattice generated by $a',b'$ in $\Z^2$ (its determinant) equals $3 (s_{d-1}-1)^2$. As this also holds for $a',c'$ and $b',c'$, we see that $P'$ has area $9 (s_{d-1}-1)^2$. We set 
\[T' := P'/(s_{d-1}-1).\]
We define $\Delta_2 := \conv\{0,e_1,e_2\}\subset \R^2$. 

\begin{lemma}
In this notation, $T'$ is a lattice triangle with $T' \cong 3 \Delta_2$.
\label{lemma3}
\end{lemma}

\begin{proof}
Assume $T' \not\cong 3 \Delta_2$. As $T'$ has area $9$ and barycenter $0$, it follows from the proof of Ehrhart's conjecture in dimension two \cite{Ehr55} that $T'$ contains an interior lattice point $z' \not=0$. Now, let $y'$ be the unique point in the boundary of $P'$ such that $z' = \lambda y'$ with $\lambda > 0$. Thus, $\lambda < \frac{1}{s_{d-1}-1}$. Let $y \in \conv\{a,b,c\}$ with $\pi(y) = y'$, and define $z := \lambda y$ (so $\pi(z)=z'$). Next, let us consider the following $(d-2)$-dimensional half-open cube
\begin{equation}
C := \sum_{i=1}^{d-2} \left(0,\frac{s_{d-1}-1}{s_{d-1}-2}\right] e_i.
\label{eq:C}
\end{equation}
It is straightforward to verify that its vertex $\sum_{i=1}^{d-2} \frac{s_{d-1}-1}{s_{d-1}-2} e_i$ is contained in the relative interior of the facet $\conv\{s_1 e_1, \ldots, s_{d-2} e_{d-2}\}$ of $R$. Hence, this point is also the only point of $C$ that is not contained in $\relint(R)$. Shrinking $C$ by $(1-\lambda)$ we get a half-open cube $(1-\lambda) C$ of edge lengths
\[(1-\lambda) \left(\frac{s_{d-1}-1}{s_{d-1}-2}\right) > \left(1-\frac{1}{s_{d-1}-1}\right)\left(\frac{s_{d-1}-1}{s_{d-1}-2}\right) = 1.\]
In particular, any translation in $\R^{d-2}$ by a vector in $\R^{d-2}$ must contain a lattice point in its interior. Therefore, as the two last coordinates of $z$ are integers, $z + (1-\lambda)C = \lambda y + (1-\lambda) C \subset P$ must contain a lattice point in its relative interior. This lattice point is by construction in the interior of $P$ and does not project to $0$, so it is different from $p$, a contradiction.
\end{proof}

We define the triangle $T := \conv\{a,b,c\}/(s_{d-1}-1)$. 

\begin{lemma}
In this notation, $T$ is a lattice triangle whose projection onto $T'$ is an isomorphism of lattice polytopes. In particular, $T \cong 3 \Delta_2$.
\end{lemma}

\begin{proof}
Recall that $\pi$ projects $T$ bijectively onto $T'$ and maps lattice points to lattice points. We will show that for any lattice point in $T'$ the unique preimage in $T$ is also a lattice point. As $T'$ is a lattice triangle and contains a lattice basis, this implies the statement. So, let $w'$ be a non-zero lattice point in $T'$. As $T' \cong 3 \Delta_2$, $w'$ is on the boundary of $T'$. We follow the proof of the previous lemma. We set $\lambda := \frac{1}{s_{d-1}-1}$. Again, we define $C$ as in $\eqref{eq:C}$ and observe that $(1-\lambda) C = \sum_{i=1}^{d-2} (0,1] e_i$. Let $w = (w_1, \ldots, w_{d-2}, w') \in \conv\{a,b,c\}$ with $\pi(w) = w'$. As above, $w+(1-\lambda) C$ must contain a lattice point $q$. As observed in above proof, $(w_1+1, \ldots, w_{d-2}+1,w')$ is the only point of $w+(1-\lambda) C$ that is not in $\relint(P)$. Hence, $(w_1+1, \ldots, w_{d-2}+1,w')=q \in N$, thus $w \in N$, as desired.
\end{proof}

Recall from above that $\Z e_1 + \cdots + \Z e_{d-2}$ is affinely spanned by $p, s_1 e_1, \ldots, s_{d-2} e_{d-2}$ and that the index in $N$ of the sublattice affinely spanned by $p, s_1 e_1, \ldots, s_{d-2} e_{d-2}, a,b$ equals~$3 (s_{d-1}-1)^2$ (in other words, the simplex with these vertices has volume $3 (s_{d-1}-1)^2$). Hence, the volume of the simplex with vertices $p, s_1 e_1, \ldots, s_{d-2} e_{d-2}, a/(s_{d-1}-1),b/(s_{d-1}-1)$ equals $3$. Now, since $T \cong 3 \Delta_2$, we have $w := \frac{2}{3} a/(s_{d-1}-1) + \frac{1}{3} b/(s_{d-1}-1) \in N$, thus this implies that $p, s_1 e_1, \ldots, s_{d-2} e_{d-2}, a/(s_{d-1}-1),w$ is an affine lattice basis of $N$ (as it has volume $1$). 
From the relations $p = \left(\sum_{i=1}^{d-2} \frac{s_i e_i}{s_i}\right) +\frac{1}{s_{d-1}-1} 0$ and $a+b+c=0$ we deduce that fixing this lattice basis fixes all vertices of $P$. Hence, $P$ is uniquely determined up to isomorphisms.

By \cite[Proposition~4.4, Definition~3.4]{Nil07} the convex hull of $\{0,s_1 e_1,\ldots,s_{d-2} e_{d-2},3 (s_{d-1}-1) e_{d-1},3 (s_{d-1}-1) e_d\}-(1,\ldots,1)$ is a reflexive simplex with multiplicity $3 (s_{d-1}-1)^2$. Hence, $P$ must be isomorphic to this simplex as claimed in Theorem~\ref{conj}.

$\hfill \qed$

\begin{acknowledgments}
GA and BN are~PIs in the Research Training Group Mathematical Complexity Reduction funded by the Deutsche Forschungsgemeinschaft (DFG, German Research Foundation) - 314838170, GRK 2297 MathCoRe. AK is supported by EPSRC Fellowship~EP/N022513/1. In the master thesis of~ML the upper bound on the multiplicity was determined in the case of reflexive simplices.
\end{acknowledgments}

\bibliographystyle{amsalpha}
\bibliography{bibliography}
\end{document}